\documentclass[11pt]{amsart}
\usepackage{fullpage}
\usepackage{graphicx}
\usepackage{amssymb,verbatim}
\usepackage{epstopdf}
\usepackage{ amssymb, amsmath, amsthm, amsfonts,mathrsfs,verbatim}
\usepackage{enumerate}
\usepackage{algorithm}
\usepackage[noend]{algpseudocode}
\usepackage{cleveref}
\usepackage{setdeck}
\usepackage{array}
\usepackage{multirow}
\newcommand{\Peak}{\text{Peak}}

\newtheorem{theorem}{Theorem}[section]
\newtheorem{proposition}[theorem]{Proposition}
\newtheorem{lemma}[theorem]{Lemma}

\theoremstyle{definition}
\newtheorem{definition}[theorem]{Definition} 
\newtheorem{example}[theorem]{Example}

\newtheorem*{solution*}{Solution}

\theoremstyle{remark}
\newtheorem{remark}[theorem]{Remark}
\title{Metrics on permutations with the same peak set}
\author{Alexander Diaz-Lopez, Kathryn Haymaker, Kathryn Keough, Jeongbin Park, Edward White} 
\begin{document}

\maketitle

\begin{abstract}
 Let $S_n$ be the symmetric group on the set $\{1,2,\ldots,n\}$. Given a permutation $\sigma=\sigma_1\sigma_2 \cdots \sigma_n \in S_n$, we say it has a peak at index $i$ if $\sigma_{i-1}<\sigma_i>\sigma_{i+1}$. Let $\Peak(\sigma)$ be the set of all peaks of $\sigma$ and define $P(S;n)=\{\sigma\in S_n\, | \,\Peak(\sigma)=S\}$. In this paper we study the Hamming metric, $\ell_\infty$-metric, and Kendall-Tau metric on the sets $P(S;n)$ for all possible $S$, and determine the minimum and maximum possible values that these metrics can attain in these subsets of $S_n$.    
\end{abstract}

\section{Introduction}
In this article we look at various sets of permutations and measure maximum and minimum distances between the elements in these sets. Let $S_n$ be the symmetric group, that is, the set of $n!$ symmetries of $\{1,2,\ldots, n\}$. We write the elements of $S_n$ in one-line notation, so for $\sigma\in S_n$  we write $\sigma=\sigma_{1}\sigma_{2}\cdots\sigma_{n}$ to denote the permutation that sends $1\rightarrow \sigma_{1},$ $2\rightarrow \sigma_{2},\ldots, n\rightarrow \sigma_{n}$.
We say $\sigma$ has a \textbf{peak} at position $i$ in $\{2,3,\ldots, n-1\}$ if $\sigma_{i-1}<\sigma_{i}>\sigma_{i+1}$, that is, $\sigma_i$ is greater than its two neighbors.  We define the \textbf{peak set} of $\sigma$, $\Peak(\sigma)$, to be the set of all indices at which $\sigma$ has a peak. For example, if $\sigma=58327164 \in S_8$ then $\Peak(\sigma)=\{2,5,7\}$.

We can collect all permutations that have the same peak set  $S$ and define
\[P(S;n)=\{\sigma \in S_n \, | \, \Peak(\sigma)=S\}.\]
We can partition $S_n$ as a disjoint union of sets of the form $P(S;n)$ as we range through all possible peak sets $S$.  The main purpose of this article is to describe the maximum and minimum distances for each subset $P(S;n)$ under three different metrics: the Hamming metric, $\ell_\infty$-metric, and Kendall-Tau metric. We report our results in Proposition \ref{prop:minPS} and Theorems \ref{thm:maxK}, \ref{thm:maxell}, and \ref{thm:maxH}.

Our study was motivated by recent work on peaks of permutations. The sets $P(S;n)$ were first studied by  
Nyman in \cite{n03} to show that sums of permutations with the same peak set form a subalgebra of the group algebra of $S_n$ (over $\mathbb{Q}).$ Later, Billey, Burdzy, and Sagan \cite{bbs13} studied the cardinality of the sets $P(S;n)$ and showed 
\[|P(S;n)|=2^{n-|S|-1}p_S(n),\]
where $p_S(n)$ is a polynomial in $n$ known as the peak polynomial of $S$. The study of these polynomials has led to a flurry of work such as \cite{bft16,dhio17,dhip17,gg21,o20}.

Permutations can be used to rank a collection of objects or quantities, and different notions of distances between pairs of permutations have been studied extensively \cite{dh98, kg90}. More recent applications of permutations include data representation, for example in flash memory storage. In the context of data representation,  the Hamming metric, the $\ell_{\infty}$   metric, and the Kendall-Tau metric have all been considered \cite{bm10,ck69,kltt10}. 

\section{Metrics on $S_n$}
In this section we will formally define the metrics we use to measure the distance between two permutations. First we recall the definition of a metric. 
Given a set $S$, a metric $d$ on $S$ is a map $d:S \times S \to [0,\infty)$ such that for $\sigma,\rho,\tau \in S$,
\begin{enumerate}
    \item $d(\sigma,\rho)=0$ if and only if $\sigma=\rho$,
    \item $d(\sigma,\rho)=d(\rho,\sigma)$,
    \item $d(\sigma,\tau)\leq d(\sigma,\rho)+d(\rho,\tau)$.
\end{enumerate}
In this article, we will use three metrics: the \textit{Hamming metric},  \textit{$\ell_\infty$-metric}, and \textit{Kendall-Tau metric}.

\begin{definition}\label{def:metrics}
Let $d_H$, denoting the \textbf{Hamming metric}, be the map $d_H:S_n \times S_n \to [0,\infty)$ such that $d_H(\sigma,\rho)$ is the number of indices where $\sigma$ and $\rho$ differ. That is, if $\sigma=\sigma_1\sigma_2\ldots \sigma_n$ and $\rho=\rho_1\rho_2\ldots \rho_n$ then 
\[d_H(\sigma,\rho)=|\{i \, | \, \sigma_i\neq \rho_i\}|.\]
Let $d_\ell$, denoting the  \textbf{$\ell_\infty$-metric}, be the map $d_{\ell}:S_n \times S_n \to [0,\infty)$ such that 
\[d_{\ell}(\sigma,\rho)=\max\{|\sigma_i - \rho_i| \, | \, 1\leq i \leq n\}.\]
Let $d_K$, denoting the \textbf{Kendall-Tau metric}, be the map $d_K:S_n \times S_n \to [0,\infty)$ such that 
$d_K(\sigma,\rho)$ is the number of pairs $(i,j)$ such that $1\leq i <j \leq n$ and $(\sigma_i-\sigma_j)(\rho_i-\rho_j)<0$. The pairs $(i,j)$ counted by $d_K$ are called \textit{deranged pairs}.
\end{definition}
\begin{example}\label{ex:metrics}
Consider $\sigma,\rho \in S_5$ where $\sigma=14325$ and $\rho=25314$. Then, $\sigma$ and $\rho$ differ in four of the five entries, thus $d_H(\sigma, \rho)=4$. The differences between the indices of $\sigma$ and $\rho$ are $|1-2|, |4-5|,|3-3|,|2-1|,|5-4|,$ thus $d_\ell(\sigma, \rho)=1$. Finally, out of the 10 possible pairs $(i,j)$ with $1\leq i<j\leq 5$, only $(1,4)$ and $(2,5)$ satisfy that $(\sigma_i-\sigma_j)(\rho_i-\rho_j)<0$, hence  $d_K(\sigma, \rho)=2$.

\end{example}

It is worth noting that the Kendall-Tau metric has an alternative description which is helpful in some contexts. For permutations $\sigma,\rho \in S_n$, let $d_K'(\sigma,\rho)$ be the minimum number of swaps of the form $(i,i+1)$ that transform $\sigma$ into $\rho$, that is, $d_K'(\sigma,\rho)$ is the minimum number $n$ such that there exist transpositions $\tau_1,\ldots, \tau_n$ of the form $(i,i+1)$ with $\tau_n\cdots\tau_1\sigma=\rho$. In Proposition \ref{prop:KTequiv} we show that $d_K(\sigma,\rho)=d'_K(\sigma,\rho)$, for all $\sigma,\rho\in S_n$. For example, for the permutations $\sigma=14325$ and $\rho=25314$ in Example \ref{ex:metrics}, we can swap $1$ and $2$ and then swap $4$ and $5$ to convert $\sigma$ into $\rho$.

We now present two lemmas, one about $d_K$ and one about $d'_K$, that will be helpful in proving Proposition \ref{prop:KTequiv}. 

\begin{lemma}\label{lem:KTrightinvariant}
The Kendall-Tau metric is right invariant, that is, for any $\sigma,\tau, \alpha\in S_n,$ $d_K(\sigma, \rho)=d_K(\sigma\alpha, \rho\alpha)$.
\end{lemma}
\begin{proof}
    Let $(i,j)$ be any pair such that $1\leq i <j\leq n$. Consider the pair $(\alpha^{-1}(i), \alpha^{-1}(j))$ or $(\alpha^{-1}(j), \alpha^{-1}(i))$, whichever has the first entry greater than the second entry. Without loss of generality, assume it is $(\alpha^{-1}(i), \alpha^{-1}(j))$. Then, 
     \begin{equation}\label{eq:rightinv}(\sigma\alpha_{\alpha^{-1}(i)} -\sigma\alpha_{\alpha^{-1}(j)})(\rho\alpha_{\alpha^{-1}(i)} - \rho\alpha_{\alpha^{-1}(j)}) = (\sigma_i-\sigma_j)(\rho_i - \rho_j).  \end{equation}
     Thus, if $(i,j)$ is a deranged pair for $(\sigma,\rho)$ then $(\alpha^{-1}(i), \alpha^{-1}(j))$ is a deranged pair for $(\sigma\alpha, \rho\alpha)$. Similarly, if $(i,j)$ is not a deranged pair for $(\sigma,\rho)$ (meaning $(\sigma_i-\sigma_j)(\rho_i - \rho_j)>0$) then by Equation \eqref{eq:rightinv} neither $(\alpha^{-1}(i), \alpha^{-1}(j))$ nor $(\alpha^{-1}(j), \alpha^{-1}(i))$ are deranged pairs for $(\sigma\alpha, \rho\alpha)$. Since both $(\sigma,\rho)$ and $(\sigma\alpha, \rho\alpha)$ have the same number of deranged pairs, then $d_K(\sigma, \rho)=d_K(\sigma\alpha, \rho\alpha)$.
\end{proof}
\begin{lemma}\label{lem:d_K'}
For $\sigma,\rho,\alpha \in S_n$, the following statements hold:
\begin{enumerate}[(a)]
    \item $d_K'(\sigma,\rho)=d_K'(\sigma\alpha,\rho\alpha)$,\label{part:rightinv}
    \item $d_K'(\sigma,\rho)=d_K'(\rho,\sigma)$,\label{part:sym}
\end{enumerate}
\end{lemma}
\begin{proof}
Let $\tau_1,\tau_2,\ldots, \tau_n$ be any  collection of transpositions of the form $(i,i+1)$ that transforms $\sigma$ into $\rho$, that is,  $\tau_n\cdots\tau_2\tau_1\sigma=\rho.$ Multiplying by $\alpha$ on both sides we get  $\tau_n\cdots\tau_2\tau_1\sigma\alpha=\rho\alpha,$ thus $d_K'(\sigma,\rho)\geq d_K'(\sigma\alpha,\rho\alpha)$. Similarly, let  $\tau_1',\tau_2',\ldots, \tau_m'$ be any collection of transpositions of the form $(i,i+1)$ that transforms $\sigma\alpha$ into $\rho\alpha$, that is, $\tau_n'\cdots \tau_2'\tau_1'\sigma\alpha=\rho\alpha$. Multiplying by $\alpha^{-1}$ on the right we get $\tau_n'\cdots \tau_2'\tau_1'\sigma=\rho$. Thus, $d_K'(\sigma,\rho)\leq d_K'(\sigma\alpha,\rho\alpha)$, which completes the proof of part ($\ref{part:rightinv}$).

Part ($\ref{part:sym}$) follows from the fact that  for any collection $\tau_1,\tau_2,\ldots, \tau_n$ of transpositions of the form $(i,i+1)$ we have that if $\tau_n\cdots \tau_1\sigma=\rho$ then $ \sigma=\tau_1\cdots \tau_n\rho$. Thus, the minimum number of swaps of the form $(i,i+1)$ that transforms $\sigma$ into $\rho$ is the minimum number of swaps of the form $(i,i+1)$ that transforms $\rho$ into $\sigma$.
\end{proof}

We are now ready to prove that both $d_K$ and $d'_K$ are the same metric.

\begin{proposition} \label{prop:KTequiv}
For $\sigma, \rho \in S_n$, the value $d_K(\sigma, \rho)=d_K'(\sigma,\rho)$.
\end{proposition}

 \begin{proof} 
 Theorem 1 in ~\cite{ck69} shows that for any permutation $\tau\in S_n,$ we have $d_K(\tau,e)=d_K'(\tau,e)$, where $e$ is the identity permutation. This, together with Lemmas \ref{lem:KTrightinvariant} and \ref{lem:d_K'}, imply 
\[d_K(\sigma, \rho)=d_K(e, \rho\sigma^{-1})=d_K( \rho\sigma^{-1},e)=d_K'( \rho\sigma^{-1},e)=d_K'(e, \rho\sigma^{-1})=d_K'(\sigma, \rho).\qedhere\] 

\end{proof}

\begin{remark}
The definition of the Kendall-Tau metric varies among different sources, although most often the definitions  are equivalent. 
For example, Diaconis (p.112, \cite{d88}) defines the Kendall-Tau distance between permutations $\pi$ and $\sigma$  as follows: 
\[ I(\pi, \sigma)  = \text{minimum number of pairwise adjacent transpositions taking } \pi^{-1} \text{ to }\sigma^{-1}.\] 
This definition is equivalent to the one we present in  Definition~\ref{def:metrics} and $d'_K(\sigma, \rho)$. 
The same definition appears in \cite{dg77}, and is named after Kendall based on work in the 1930's and beyond \cite{kg90}. Some ambiguity arises since Kendall defined a metric on rankings, and rankings can be transformed into permutations in two different ways. A non-equivalent definition of Kendall-Tau distance between permutations is often used in rank modulation applications in the area of coding for flash memory storage (see, for example \cite{bm10}). 
\end{remark} 

For a given set $S$ of permutations, we will consider the pair-wise distances between distinct permutations in the set as well as the maximum and minimum values attained. 

\begin{definition} For a metric $d$ on a set $S$, let $d(S)$ be the set of positive integers defined as follows: \[ d(S)=\{d(\sigma, \rho)|\sigma, \rho\in S, \sigma\neq \rho\}.\]
We will denote the minimum and maximum of the set $d(S)$ as $\min(d(S))$ and $\max(d(S))$, respectively.  \end{definition}
When $S=S_n$, it is straightforward to compute the values of $\min(d(S))$ and $\max(d(S))$ for the Hamming, $\ell_\infty$, and Kendall-Tau metrics, as we show in Proposition \ref{prop:Sn}. In Section \ref{sec:maxminPeakSet}, we consider the same question for 
subsets of $S_n$ defined by their common peak set.
\begin{proposition} \label{prop:Sn}
For $S_n$ with $n\geq2$, the minimum and maximum for each of the three metrics in Definition \ref{def:metrics} are
\begin{itemize}
    \item $\min(d_H(S_n))=2$, $\max(d_H(S_n))=n$
    \item  $\min(d_\ell(S_n))=1$, $\max(d_\ell(S_n))=n-1$
    \item $\min(d_K(S_n))=1$, $\max(d_K(S_n))= \binom{n}{2}$.
\end{itemize}
\end{proposition}
\begin{proof}
For the Hamming metric, the minimum possible value $\min(d_H(S_n))$ is 2 because distinct permutations must differ in at least 2 indices. This minimum distance is achieved by, e.g., the pair $\sigma=123\cdots n$ and $\rho=213\cdots n$. The maximum distance occurs when all indices of $\sigma$ and $\rho$ are different, for example, with $\sigma=12\cdots n$ and $\rho=2 \, 3 \, \cdots n \, 1$, so $\max(d_H(S_n))=n$.

For the $\ell_\infty$-metric, the minimum possible value of $\min(d_\ell(S_n))$ is 1, which is achieved by, e.g., the pair $\sigma=12\cdots n$ and $\rho=213\cdots n$. The maximum possible value of $\max(d_\ell(S_n))$ would be $n-1$, which occurs when $\sigma_i=n$ and $\rho_i=1$ (or vice-versa) for some index $i$. The pair $\sigma=12\cdots n$ and $\rho=n \, n-1\, \cdots 1$ achieves this maximum. 

For Kendall-Tau metric, the minimum distance occurs when the least number of pairs $(i,j)$ such that $1\leq i <j \leq n$ and $(\sigma_i-\sigma_j)(\rho_i-\rho_j)<0$ is obtained. The least number of pairs $(i,j)$ possible is 1, and this occurs when $\sigma=12\cdots n$ and $\tau=213\cdots n$. The maximum distance occurs when all $\binom{n}{2}$ pairs $(i,j)$ with $1\leq i <j \leq n$ satisfy $(\sigma_i-\sigma_j)(\rho_i-\rho_j)<0$. This happens when $\sigma=12\cdots n$ and $\rho=n \, n-1 \, \cdots 1$.
\end{proof}

\section{Maximum and minimum distances among permutations with the same peak set} \label{sec:maxminPeakSet}

For the remainder of this paper, we will explore the maximum and minimum values of the three metrics described in Definition \ref{def:metrics} in sets of permutations with the same peak set. Recall that for any set $S\subseteq [n]$ of indices
\[P(S;n)=\{\sigma \in S_n \, | \, \Peak(\sigma)=S\}.\]
We say $S$ is \textbf{admissible} if $P(S;n)\neq \emptyset$.

In Proposition \ref{prop:minPS} we explore $\min(d_K(P(S;n)),\min(d_\ell(P(S;n)),$ and $\min(d_H(P(S;n))$ for admissible sets $S$ and in Theorems \ref{thm:maxK}, \ref{thm:maxell}, and \ref{thm:maxH} we explore the equivalent problem for maximum values. The following lemma is useful for subsequent results.
\begin{lemma}[{\cite[Lemma 4.4]{s05}}]\label{lem:swap}
Let $S$ be an admissible set and $\sigma \in P(S;n)$. For any $i \in \{2,3,\ldots, n-1\}$, if $i$ and $i+1$ do not appear consecutively in $\sigma$ then swapping $i$ and $i+1$ creates a permutation $\sigma'$ with the same peak set as $\sigma$, i.e., $\sigma' \in P(S;n)$. If $i=1$ then swapping $1$ and $2$ will produce a permutation with the same peak set as $\sigma$.
\end{lemma}

\begin{proposition}\label{prop:minPS}
Given an admissible peak set $S$ and $P=P(S;n)$ for $n\geq 2$, we have
\[\min(d_H(P))=2,\quad\min(d_\ell(P))=1,\text{ and  }\min(d_K(P))=1.\] 
\end{proposition}
\begin{proof}
For any set $S$, by definition we have that $P(S;n)\subseteq S_n$. Hence, by Proposition \ref{prop:Sn} the minimum value of  $\min(d_H(P))$ is at least 2 and for both $\min(d_\ell(P))$ and $\min(d_K(P))$ it is at least 1. Hence, it is enough to find pairs of permutation in $P$ that attain these values. 

Let $S$ be admissable and $\sigma$ be any permutation in $P(S;n)$. By Lemma \ref{lem:swap}, swapping $1$ and $2$ in $\sigma$ will lead to a permutation $\sigma'$ with the same peak set as $\sigma$. Since $\sigma$ and $\sigma'$ only differ in the indices where $1$ and $2$ are located, $d_H(\sigma,\sigma')=2$. Using the same $\sigma$ and $\sigma'$ we see that $d_\ell(\sigma,\sigma')=1$ as the only indices in which they differ have entries $1$ and $2$ and $|2-1|=|1-2|=1$. Finally, for the Kendall-Tau metric, we get that $d_K(\sigma,\sigma')=1$ as the only deranged pair between $\sigma$ and $\sigma'$ is the pair of indices where $1$ and $2$ are located.
\end{proof}

We now proceed to explore the maximum values of the metrics when restricted to sets of permutations with the same peak set. Proposition \ref{prop:Sn} bounds the values of $\max(d_H(P(S;n)))$, $\max(d_\ell(P(S;n)))$, and $\max(d_K(P(S;n)))$ by $n,n-1,$ and $\binom{n}{2}$, respectively, as $P(S;n)\subseteq S_n$. In the main results of this section, Theorems \ref{thm:maxK}, \ref{thm:maxell}, and \ref{thm:maxH}  we show that these values are not always attained in the sets $P(S;n)$. Throughout the next results, we will use two particular permutations as our starting point to create others. 

\begin{definition}
\label{def:e} 
Let $\mathbf{e}$ be the identity permutation $\mathbf{e}=1\,2\cdots n-1\, n$ and $\mathbf{e}^*=n\, n-1\cdots 2\,1$. For an admissible peak set $S$, define $\mathbf{e}[S]$ as the permutation obtained by swapping the entries of $k$ and $k+1$ in $\mathbf{e}$ for each $k \in S$. Similarly, let $\mathbf{e}^*[S]$ be the permutation obtained by swapping the entries of $k-1$ and $k$ in $\mathbf{e}^*$, for each $k\in S$. Since any admissible set $S$ has no consecutive entries, these permutations are well-defined as the order of the swaps does not matter. More explicitly, we have that 
for $i\in \{1, 2, \ldots, n\}$,
\[\textbf{e}[S]_i=\begin{cases} 
i+1& \text{ if } i \in S  \\
i-1 &\text{ if } i \in  \{s+1 \, | \, s \in S\}\\
i &  \text{ otherwise,} \end{cases}
\text{ and }
\textbf{e}^*[S]_i=\begin{cases} 
(n+1-i)+1& \text{ if } i \in S  \\
(n+1-i)-1 &\text{ if } i \in  \{s-1 \, | \, s \in S\}\\
n+1-i &  \text{ otherwise.} \end{cases}\]
\end{definition}

For example, for the set $S=\{2,5,7\}$ and $S_9$, we have $\textbf{e}[S]=132465879$ and $\textbf{e}^*[S]=897563421.$

\begin{theorem}\label{thm:maxK}
For $n\geq 2$, the maximum Kendall-Tau distance between permutations in $P(S;n)$ is $\binom{n}{2}-2|S|$.
\end{theorem}
\begin{proof} 
For any pair of permutations $\sigma$ and $\rho$ in $P(S;n)$, we have that $\sigma_{i-1}<\sigma_i>\sigma_{i+1}$ for each $i \in S$, and analogously for $\rho$. Therefore, the pairs $(i-1, i)$ and $(i, i+1)$ are not deranged pairs for $\sigma,\rho$ since
\[ (\sigma_{i-1}-\sigma_{i})(\rho_{i-1}-\rho_{i})>0 \;\text{ and }\;(\sigma_{i}-\sigma_{i+1})(\rho_{i}-\rho_{i+1})>0.\]
Since there are a total of $\binom{n}{2}$ pairs of possible deranged pairs, we have  that $d_{K}(\sigma, \rho)\leq \binom{n}{2}-2|S|$.

We now consider a pair of permutations that attain the bound  $\binom{n}{2}-2|S|$. First, note that the permutations $\textbf{e}$ and $\textbf{e}^*$ are Kendall-Tau distance $\binom{n}{2}$ apart since for all index pairs $(i,j)$ with $1\leq i<j\leq n$, $\textbf{e}$ has $(\textbf{e}_i-\textbf{e}_j)<0$ while $\textbf{e}^*$ has $(\textbf{e}^*_i-\textbf{e}^*_j)>0$. Consider $\textbf{e}[S]$ and $\textbf{e}^*[S]$ as defined in Definition~\ref{def:e}. We claim $d_K(\mathbf{e}[S],\mathbf{e}^*[S])=\binom{n}{2}-2|S|$.

Since every pair of indices $(i,j)$ with $1\leq i<j\leq n$ was a deranged pair for $\mathbf{e}$ and $\mathbf{e}^*$ and we only altered the (consecutive) entries in indices $k-1$ and $k$ in $\mathbf{e}[S]$, and (consecutive) entries in indices $k$ and $k+1$ in $\mathbf{e}^*[S]$ for $k \in S$, then only the pairs $(k-1,k)$, and $(k, k+1)$ might no longer be deranged pairs for  $\mathbf{e}[S]$ and $\mathbf{e}^*[S]$. Indeed, for the pair $(k, k+1)$,
\begin{align*}
     (\mathbf{e}[S]_k-\mathbf{e}[S]_{k+1})&(\mathbf{e}^*[S]_k-\mathbf{e}^*[S]_{k+1})\\
     &=(k+1 - k)((n+1-k)+1-(n+1-(k+1))=2>0.  
\end{align*}
Similarly, for the pair $(k-1, k)$ we have 
\begin{align*}
    (\mathbf{e}[S]_{k-1}-\mathbf{e}[S]_k)&(\mathbf{e}^*[S]_{k-1}-\mathbf{e}^*[S]_k)\\
    &= (k-1-(k+1))(n+1-(k-1)-1-((n+1-k)+1))=2>0.\end{align*}
Thus, for each peak in $S$ we have two pairs that are not deranged, hence $d_K(\mathbf{e}[S],\mathbf{e}^*[S])=\binom{n}{2}-2|S|$. 

\end{proof}

\begin{theorem}\label{thm:maxell}
For $n\geq 2$, the maximum $\ell_\infty$ distance between permutations in $P(S;n)$ is $n-2$ when $S$ contains peaks at indices $2$ and $n-1$, and $n-1$ otherwise.
\end{theorem}
\begin{proof}
 First consider the case when  $\{2,n-1\} \not\subseteq S$. If $2\notin S$, then the permutations $\mathbf{e}[S]$ and $\mathbf{e}^*[S]$ achieve the maximum $\ell_\infty$-distance $n-1$ as $\mathbf{e}[S]_1=1$ and $\mathbf{e}^*[S]_1=n$. Similarly, if $n-1\not \in S$ then $\mathbf{e}[S]_n=n$ and $\mathbf{e}^*[S]_n=1$, and therefore $d_{\ell}(\mathbf{e}[S], \mathbf{e}^*[S])=n-1$. 

If $\{2, n-1\}\subseteq S$, then we first claim $d_\ell(\sigma,\rho)\leq n-2$ for any pair of distinct permutations $\sigma,\rho \in P(S;n)$. In any permutation in $P(S;n)$, $n$ must not appear in index $1$ nor index $n$ since indices $2$ and $n-1$ are peaks. Since $n$ is larger than any other number in the permutation, it must be in an index that is a peak. On the other side, $1$ will never appear in an index that is a peak. Hence, $d_\ell(\sigma,\rho)$ cannot be $n-1$ as the only way to obtain this would be for $n$ and $1$ to appear in the same index in $\sigma$ and $\rho$, respectively. Thus, $d_\ell(\sigma,\rho)\leq n-2$. 
To show that this bound is achieved, consider the permutations $\mathbf{e}[S]$ and $\mathbf{e}^*[S]$. Since $\mathbf{e}[S]_1=1$ and $\mathbf{e}^*[S]_1=n-1$, then $d_\ell(\mathbf{e}[S],\mathbf{e}^*[S])=n-2$.
\end{proof}

The next result considers the maximum Hamming distance between permutations with the same peak set in $S_n$ for $n\geq 4$. We first remark that for $n=2$, the only peak set is $\emptyset$ and $\max d_H(P(\emptyset;2))=2$, and for $n=3$, we have that $\max d_H(P(\emptyset;3))=3$ and $\max d_H(P(\{2\};3))=2$.
\begin{theorem}\label{thm:maxH}
For $n\geq 4$ and any admissible peak set $S$, the maximum Hamming distance between  permutations in $P(S;n)$  is $n$.  
\end{theorem}

\begin{proof}
We proceed by induction on $n$. For the base cases of $n=4$ and $n=5$, consider the  pairs of permutations in each of the admissible peak sets shown in Tables~\ref{tab:base-case-4} and \ref{tab:base-case-5}, respectively. Suppose that for every $4\leq j<n$ the maximum Hamming distance between permutations in $P(S;j)$ is $j$.  

\begin{table}[ht]
    \centering
    \begin{tabular}{c|c|c}
        $S=\emptyset$ & $S=\{2\}$ & $S=\{3\}$   \\ \hline
        1\,2\,3\,4 & 1\,3\,2\,4 & 1\,3\,4\,2 \\
        4\,3\,2\,1 & 2\,4\,3\,1 & 4\,2\,3\,1 
    \end{tabular}
    \caption{Pairs of permutations in $S_4$ with the same peak set and Hamming distance four.}
    \label{tab:base-case-4}
\end{table}

\begin{table}[t]
    \centering
    \begin{tabular}{c|c|c|c|c}
        $S=\emptyset$ & $S=\{2\}$ & $S=\{3\}$  & $S=\{4\}$ & $S=\{2, 4\}$ \\ \hline
        1\,2\,3\,4\,5 & 1\,3\,2\,4\,5 & 1\,3\,4\,2\,5 &4\,3\,2\,5\,1& 1\,3\,2\,5\,4\\
        5\,3\,2\,1\,4 & 2\,5\,3\,1\,4 & 5\,2\,3\,1\,4&5\,4\,1\,3\,2 & 4\,5\,1\,3\,2 
    \end{tabular}
    \caption{Pairs of permutations in $S_5$ with the same peak set and Hamming distance five.}
    \label{tab:base-case-5}
\end{table}

\begin{table}[t]
    \centering
    \begin{tabular}{c|c|c|c|c|c|c|c}
        $S=\emptyset$ & $S=\{2\}$ & $S=\{3\}$  & $S=\{4\}$ & $S=\{5\}$ &$S=\{2, 4\}$&$S=\{2, 5\}$&$S=\{3, 5\}$\\ \hline
        1\,2\,3\,4\,5\,6 & 1\,3\,2\,4\,5\,6 & 1\,3\,4\,2\,5\,6 &4\,3\,2\,5\,1\,6&1\,2\,3\,4\,6\,5 &1\,3\,2\,5\,4\,6&1\,3\,2\,4\,6\,5&1\,3\,4\,2\,6\,5\\
        6\,3\,2\,1\,4\,5 & 2\,6\,3\,1\,4\,5 & 6\,2\,3\,1\,4\,5&6\,4\,1\,3\,2\,5 & 6\,3\,2\,1\,5\,4&4\,6\,1\,3\,2\,5&2\,6\,3\,1\,5\,4&6\,2\,3\,1\,5\,4 
    \end{tabular}
    \caption{Pairs of permutations in $S_6$ with the same peak set and Hamming distance six, created using the constructions in the proof of Theorem \ref{thm:maxH}.}
    \label{tab:base-case-6}
\end{table}

Let $S$ be an admissible peak set for permutations in $S_n$, and for this case assume $n-1\notin S$. 
Since $n-1\notin S$,  $S$ is also an admissible peak set for permutations in $S_{n-1}$, so there exist permutations $\sigma, \rho\in P(S; n-1)$ such that $d_H(\sigma, \rho)=n-1$ by the inductive hypothesis. Since $\sigma$ and $\rho$ differ in every index, in at least one of the permutations $n-1$ does not appear in index $n-1$. Without loss of generality, assume $\rho_{n-1}\neq n-1.$  Construct permutations $\sigma', \rho'$  in $S_n$ as follows: 
$\sigma'$ equals $\sigma$ with $n$ appended at the end. For $\rho'$, first form an intermediate permutation $\rho''$ by appending $n$ to the end of $\rho$. Then to obtain $\rho'$, swap values $n$ and $n-1$ in $\rho''$. We claim that $d_H(\sigma', \rho')=n$ and the peak set of both $\sigma'$ and $\rho'$ is $S$. 

First recall that since $d_H(\sigma, \rho)=n-1$, we have that $d_H(\sigma', \rho'')=n-1$ since they are formed by appending $n$ to the end of each permutation. Swapping the values $n-1$ and $n$ in $\rho''$ results in the distance $d_H(\sigma', \rho')=n.$ The peak set of $\sigma'$ and $\rho''$, $S$, is inherited from $\sigma$ and $\rho$ by construction. Since $\rho_{n-1}\neq n-1$ then $n-1$ and $n$ are not neighbors in $\rho''$. By Lemma \ref{lem:swap} the peak set of $\rho'$ is the same as the peak set of $\rho''$, which is $S$. 

Now assume $S$ is an admissible peak set for permutations in $S_n$, and $n-1\in S$. Define $S'=S\setminus \{n-1\}$, which is an admissible peak set on $S_{n-2}$. By our inductive assumption there exist permutations $\sigma, \rho\in P(S';n-2)$ such that $d_H(\sigma, \rho)=n-2$. Thus, at least one of $\sigma$ or $\rho$ must have its $n-2$ index not equal to $n-2$. Without loss of generality, suppose $\rho_{n-2}\neq n-2$. Define the following permutations in $S_n$: $\sigma'$ equals $\sigma$ with values $n$ and $n-1$ appended to the end, in that order, that is, 
\[\sigma'=\sigma_1\cdots \sigma_{n-2}\, n\, n-1.\]
 Starting with $\rho$, define $\rho''$ to be the permutation $\rho$ with $n$ and $n-1$ appended to the end in that order. Let $i$ be the index such that $\rho_i=n-2$, then  $\rho''$ is of the form
\[\rho''=\rho_1\cdots \rho_{i-1}\, n-2\, \rho_{i+1} \cdots\rho_{n-2} \, n \, n-1. \]
Let $\rho'$ be 
\[\rho'=\rho_1\cdots \rho_{i-1}\, n\, \rho_{i+1} \cdots\rho_{n-2} \, n-1 \, n-2. \]
In other words, $\rho'$ equals $\rho$ with value $n-2$ replaced by $n$, and then $n-1$, $n-2$ appended in that order to the end of the permutation. By construction, $\sigma'$ and $\rho'$ differ in every index, so $d_H(\sigma', \rho')=n$. Finally, the peak set of both $\sigma'$ and $\rho'$ is $S$ as we have introduced a peak at $n-1$ and have not altered any other entry other than cyclically permuting $(n,n-1,n-2)$ in $\rho''$, which does not change the peak set. Hence, the result is proven. Table \ref{tab:base-case-6} showcases these constructions for the case $n=6$.
\end{proof}

\section{Acknowledgements}
The authors thank Villanova's Co-MaStER program. A. Diaz-Lopez's research is supported in part by National Science Foundation grant DMS-2211379.
\bibliographystyle{amsplain}
\bibliography{Reference}
\end{document}